\documentclass[12pt]{amsart}

\usepackage[utf8]{inputenc}
\usepackage{amsmath}
\usepackage{amsthm}
\usepackage{amsfonts}
\usepackage{amssymb}
\usepackage{tabu}
\usepackage{latexsym}
\usepackage[margin=1in]{geometry}
\usepackage{enumerate}
\usepackage[colorlinks=true]{hyperref}
\usepackage{mathrsfs}
\usepackage{tikz}
\usepackage[parfill]{parskip}
\usetikzlibrary{matrix,arrows,decorations.pathmorphing}
\usepackage{tikz-cd}
\usepackage[all]{xy}
\usepackage{setspace}
\usepackage{tabularx}
\usepackage{xcolor}
\usepackage{framed}

\newcommand{\Ext}{\mathrm{Ext}}

\newcommand{\Hom}{\mathrm{Hom}}
\newcommand{\End}{\mathrm{End}}
\newcommand{\sHom}{\underline{\mathrm{Hom}}}
\newcommand{\sEnd}{\underline{\mathrm{End}}}

\newcommand{\ann}{\mathrm{ann}}
\newcommand{\sann}{\underline{\mathrm{ann}}}

\newcommand{\arann}{\mathrm{ARann}}

\newcommand\ca{\mathrm{ca}}
\newcommand\co{\mathrm{co}}
\newcommand{\m}{\mathfrak{m}}

\newcommand{\module}{\mathrm{mod}}

\newcommand{\CM}{\mathrm{CM}}

\newcommand{\add}{\mathrm{add}}

\newcommand{\reflexive}{\mathrm{ref}}

\newcommand{\spec}{\mathsf{Spec}}

\newtheorem{theorem}{Theorem}[section]
\newtheorem*{theorem*}{Theorem}

\newtheorem*{lemma*}{Lemma}
\newtheorem{proposition}[theorem]{Proposition}
\newtheorem*{proposition*}{Proposition}
\newtheorem{corollary}[theorem]{Corollary}
\newtheorem*{corollary*}{Corollary}
\newtheorem{question}[theorem]{Question}
\theoremstyle{definition}
\newtheorem{definition}[theorem]{Definition}
\newtheorem*{definition*}{Definition}
\newtheorem{example}[theorem]{Example}
\newtheorem{remark}[theorem]{Remark}

\newtheorem{introtheorem}{Theorem}

\DeclareMathOperator{\tr}{\mathrm{tr}}

\title{When is the canonical conductor minimal?}
\author[\"{O}.~Esentepe]{\"{O}zg\"{u}r Esentepe}
\address{Institut für Mathematik und Wissenschaftliches Rechnen, Universität Graz,
Heinrichstraße 36, 8010 Graz, Austria}
\email{ozgur.esentepe@uni-graz.at}
\urladdr{https://www.sntp.ca}

\subjclass[2020]{13C05, 13C14, 13G05, 13H10}
\keywords{almost Gorenstein rings, nearly Gorenstein rings, far-flung Gorenstein rings, trace ideals, conductor ideals, stable annihilators, cohomology annihilators}

\begin{document}

\begin{abstract}
    For a one dimensional analytically unramified Cohen-Macaulay local ring $R$, the blowup algebra of the canonical ideal is a module finite birational extension. The conductor of this extension always contains the conductor of $R$. We study the case when there is equality. This is the case where $R$ is far from being almost Gorenstein. We characterise such rings among Arf domains and numerical semigroup rings. As an application, we compute the cohomology annihilator ideal of an explicit family of rings and show it to be equal to the conductor ideal.
\end{abstract}

\maketitle
\thispagestyle{empty}

\section{Introduction}
Let $R$ be a commutative Noetherian ring of Krull dimension $d$. The singular behaviour of the affine space $\spec R$ is determined by the homological properties of finitely generated $R$-modules under suitable conditions. More precisely, the space $\spec R$ does not have any singularities if and only if there exists a positive integer $s$ such that $\Ext_R^s(M,N) = 0$ for any pair of finitely generated $R$-modules $M,N$. In fact, one can take $s = d+1$, in this case. However, the vanishing of Ext is a binary condition. We would like to get information from our Ext-modules regarding the singularities even if they do not vanish. A precise way to do this is to look at the annihilator ideals of Ext modules. A systematic study of these annihilators has been started by Iyengar and Takahashi in \cite{iyengar-takahashi}. They defined the $n$-th \textit{cohomology annihilator ideal} $\ca^n(R)$ and the \textit{cohomology annihilator ideal} $\ca(R)$ of $R$ as
\begin{align*}
  \ca^n(R) := \bigcap_{M, N \in \module R}  \ann_R \Ext_R^n(M,N)
\end{align*}
and \begin{align*}
  \ca(R) := \bigcup_{n \geq 0} \ca^n(R)
\end{align*}
respectively and they proved under mild conditions the singular locus of $R$ is defined by $\ca(R)$. It is easy to see, using properties of syzygies and Noetherian rings, that there is an increasing chain of $n$-th cohomology annihilator ideals which stabilizes and therefore $\ca(R) = \ca^s(R)$ for $s \gg 0$. Note that we have $\ca(R) = R$ if and only if $R$ is regular

While the definition of the ideal $\ca(R)$ is fairly straightforward and can be easily motivated, computing it is not an easy task. While the interest in understanding these annihilators has been high in the recent years \cite{Bahlekeh-Salarian-Takahashi-Toosi, Liu,Esentepe-Takahashi, Akdenizli-Aytekin-Cetin-Esentepe,Kimura, Dey-Liu}, we only have a couple of families of rings where we can determine the cohomology annihilator ideal precisely: for one dimensional Gorenstein local rings \cite{Esentepe} and almost Gorenstein local rings \cite{Esentepe-almost-Gorenstein}, under suitable additional assumptions. In these cases, the cohomology annihilator ideal coincides with the conductor  ideal $\co(R)$, a well-known and well-studied ideal of a one dimensional ring. We refer to the respective papers for precise statements.

It is then natural to ask the following.
\begin{question}\label{main-question}
    Can the equality 
    \begin{align*}
        \ca(R) = \co(R)
    \end{align*}
    be generalised to other one dimensional Cohen-Macaulay local rings or if this equality holds only for rings that are close to being Gorenstein?
\end{question}
We have the following theorem as our first main theorem answering Question \ref{main-question}
\begin{introtheorem}\label{main-theorem-intro}
    There exists a family $R_{s,n}$ of one dimensional Cohen-Macaulay local rings such that for each $s \geq 2$ and $n \geq  3$
    \begin{enumerate}
        \item the ring $R_{s,n}$ is far from being almost Gorenstein, and
        \item the equality $\ca(R_{s,n}) = \co(R_{s,n})$ holds.
    \end{enumerate}
\end{introtheorem}
In the rest of this introduction section, we will make precise what we mean by being far from being almost Gorenstein and we will introduce the family $R_{s,n}$.

\subsection{Minimal canonical conductor}

Gorenstein rings are ubiquitous as Hyman Bass pointed out over half a century ago \cite{Bass}. They possess beautiful symmetry properties and they are a special family inside the broader family of Cohen-Macaulay rings. In the last few decades, there has been a growing interest in understanding Cohen-Macaulay local rings which are close to being Gorenstein. This has started with the introduction of one dimensional \textit{almost Gorenstein} rings by Barucci-Fr{\"o}berg \cite{Barucci-Froberg} and continued with the work of Goto-Matsuoka-Phuong \cite{Goto-Matsuoka-Phuong} and Goto-Takahashi-Taniguchi \cite{Goto-Takahashi-Taniguchi} in higher dimensions. Later in \cite{Herzog-Hibi-Stamate}, Herzog-Hibi-Stamate introduced \textit{nearly Gorenstein} rings and this is still a very active research area in commutative algebra. On the other extreme, Herzog-Kumashiro-Stamate \cite{Herzog-Kumashiro-Stamate} considered \textit{far-flung-Gorenstein} rings which are one dimensional rings that are not close to being Gorenstein.

The main actors in this type of study are trace ideals. Nearly Gorenstein rings and far-flung Gorenstein rings are defined by imposing conditions on the trace of the canonical module. Dao-Maitra-Sridhar showed in \cite[Proposition 7.1]{Dao-Maitra-Sridhar} that in dimension one, the almost Gorenstein property can also be seen by a condition on the trace of the blowup algebra of the canonical module. More precisely, let $R$ be a one dimensional generically Gorenstein local ring with canonical ideal $\omega$ and maximal ideal $\m$. Denote by $\tr(\omega)$ the trace of the canonical module (see Definition \ref{trace-and-conductor}) and by $b(\omega)$ the conductor of the blowup algebra $B(\omega)$ of the canonical module (see Remark \ref{trace-conductor-remark} and Remark \ref{blow-up-remarks}). Then, under mild conditions there is a chain of ideals
\begin{align*}
\co(R) \subseteq b(\omega) \subseteq \tr(\omega) \subseteq \m
\end{align*}
where $\co(R)$ is the conductor of $R$, provided that $R$ is not Gorenstein (in which case we have $b(\omega) = \tr(\omega) = R$ ). The three extreme cases have been studied:
\begin{enumerate}
    \item $R$ is a nearly Gorenstein ring if and only if the equality $\tr(\omega) = \m$ holds.
    \item $R$ is an almost Gorenstein ring if and only if the equality $b(\omega) = \m$ holds.
    \item $R$ is a far-flung Gorenstein ring if and only if the equality $\tr(\omega) = \co(R)$ holds.
\end{enumerate}
In this paper, we are treating the missing case: the case where the equality $b(\omega) = \co(R)$ holds. We say that such a ring has \textit{minimal canonical conductor}. By definition, the class of such rings contain the class of far-flung Gorenstein rings. Under some conditions, they coincide. For example, if $R$ has multiplicity 3, then $R$ has minimal canonical conductor if and only if it is far-flung Gorenstein. But in general the class of rings with minimal canonical conductor is much larger.

The theory is built around $\omega$-Ulrich modules. We refer to \cite{Dao-Maitra-Sridhar} for the definition and development of $I$-Ulrich modules. While our contribution is to make precise what it means to be \textit{far away from being almost Gorenstein}, most of the theorems in the background section follow from the work of Dao-Dey-Dutta \cite{Dao-Dey-Dutta}. We look at these foundational theorems in Section 2, where we also introduce notation and basic definitions.

In Section 3 and Section 4, we consider classical classes of examples of one dimensional rings. Firstly, in Section 3, we treat the case of Arf domains with minimal canonical conductor. We prove the following characterisation.
\begin{introtheorem}
    Let $R$ be a complete local Arf domain and assume that it is analytically unramified. The following are equivalent.
    \begin{enumerate}
        \item $R$ has minimal canonical conductor.
        \item For any birational extension $R\subseteq S \subseteq \overline{R}$ that is reflexive as an $R$-module, we have either $S = \overline{R}$ or $e(S) \geq 3$.
        \item There does not exist an $A_n$-singularity that is infinitely near to $R$.
    \end{enumerate}
\end{introtheorem}
Then, in Section 4, we look at numerical semigroup rings and prove the following characterisation.
\begin{introtheorem}
    Let $k$ be a field and $H$ be a numerical semigroup with conductor $c \geq 3$. Then, the following are equivalent.
    \begin{enumerate}
        \item The numerical semigroup ring $k[[H]]$ has minimal canonical conductor.
        \item $c-2$ is not in $H$. 
    \end{enumerate}
\end{introtheorem}
\subsection{The family of Lindo-Lyle-Maitra} Let $n,s$ be integers with $n \geq 2$ and $s \geq 1$ and define $H_{s,n}$ be the numerical semigroup generated by $n, sn+1, sn+2, \ldots, sn+n-1$ and let $k$ be a field. Then, we define $R_{s,n}$ to be the corresponding numerical semigroup ring. This ring appeared in \cite[Theorem 1.3]{Lindo-Lyle-Maitra} and they have showed that the conductor ideal $\co(R_{s,n})$ of $R_{s,n}$ has the property that 
\begin{align*}
\tr(\Omega_{R_{s,n}}\co(R_{s,n})) = \co(R_{s,n}).
\end{align*}
To conclude Theorem \ref{main-theorem-intro}, we prove that $R_{s,n}$ has minimal canonical conductor and $\Omega_{R_{s,n}}\co(R_{s,n})$ is actually isomorphic to a finite direct sum of copies of $\co(R_{s,n})$. This will be the content of Section 5 where we develop some general facts about cohomology annihilator ideals and prove Theorem \ref{main-theorem-intro}.

\subsection*{Acknowledgements} I would like to thank Alfred Geroldinger, Alessio Moscariello and Shinya Kumashiro for valuable conversations and helpful comments. I also want to thank Souvik Dey for his comments after the first preprint version of this paper and Trevor Jones for correcting a mistake in the second preprint version. The first two versions of this preprint did not talk about cohomology annihilator ideals, at all and most of the results followed relatively easily from the results of Dao-Dey-Dutta \cite{Dao-Dey-Dutta}. The central theme is changed in this version to annihilation of cohomology and we introduce and study rings with minimal canonical conductors with an application to cohomology annihilators (c.f. Question \ref{main-question}).

\section{Preliminaries}
We start this section with some conventions and general definitions.
\begin{definition}\label{trace-and-conductor}
    Let $R$ be a commutative Noetherian local ring .
     \begin{enumerate}
        \item Let $M$ be a finitely generated $R$-module. Then, the \textit{trace ideal} $\tr(M)$ of $M$ is the ideal of $R$ generated by the elements of the form $f(m)$ where $f \in M^*$ and $m \in M$.
        \item Let $Q(R)$ be the total quotient ring of $R$. By $\overline{R}$, we denote the normalisation of $R$ inside $Q(R)$. That, is $\overline{R}$ is the integral closure of $R$ inside its total quotient rings. Then, the \textit{conductor} ideal is defined as
        \begin{align*}
            \co(R) = \{r \in \overline{R} \colon r \overline{R} \subseteq R \}.
        \end{align*}
        More generally, for any birational extension $R \subseteq S$, one can define
        \begin{align*}
            \co(R,S) = \{r \in S \colon r S \subseteq R \}.
        \end{align*}
        Here, by a \textit{birational extension} we mean a ring extension $R \subseteq S \subseteq Q(R)$.
     \end{enumerate}
\end{definition}
\begin{remark}\label{trace-conductor-remark}
    Conductor ideals can be seen as trace ideals. This follows from \cite[Proposition 2.4]{Kobayashi-Takahashi-Trace} and has been carefully treated in \cite[Appendix]{Herzog-Hibi-Stamate-semigroup}. More precisely, one has $\co(R,S) = \tr(S)$.
\end{remark}
Recall that a local ring is called \textit{analytically unramified} if its completion is reduced. For an analytically unramified local ring, the normalisation $\overline{R}$ is finitely generated as an $R$-module. The converse is also true when $R$ is a one dimensional Cohen-Macaulay local ring. In this case, $R$ is reduced.
\begin{remark}
    Let $R$ be a one dimensional Cohen-Macaulay local ring which is analytically unramified. Assume that for a finitely generated $R$-module M, the trace ideal $\tr(M)$ contains a nonzerodivisor and has a principal reduction. Then, there exists an inclusion $\co(R) \subseteq \tr(M)$ \cite[Corollary 3.6]{Dao-Maitra-Sridhar}. When $R$ has infinite residue field, the principal reduction condition is satisfied automatically.
\end{remark}

\begin{definition}
    Let $R$ be a commutative Noetherian ring and $I$ a regular ideal. Then, there is a filtration of endomorphism algebras
    \begin{align*}
    R \subseteq I \colon I \subseteq I^2 \colon I^2 \subseteq \cdots I^n \colon I^n \subseteq \cdots \subseteq Q(A).
    \end{align*}
    We put
    \begin{align*}
    R^I = \bigcup_{n \geq 0} I^n \colon I^n
    \end{align*}
    and denote by $b(I)$ the conductor ideal $\co(R,R^I)$.
\end{definition}
\begin{remark}\label{blow-up-remarks}
    Assume that $R$ is a commutative Noetherian ring of dimension one and $I$ a regular ideal.
    \begin{enumerate}
        \item The ring $R^I$ coincides with the blowup algebra $B(I)$ at $I$. This was proved by Lipman in his influencial work \cite{Lipman} on Arf rings.
        \item We say that $I$ is a \textit{stable} ideal if $R^I = I \colon I$.
        \item The ideal $b(I)$ equals the trace ideal $\tr_R(R^I)$.
    \end{enumerate}
\end{remark}
Next, we will talk about $I$-Ulrich modules introduced by Dao-Maitra-Sridhar \cite{Dao-Maitra-Sridhar}. The following definition is an equivalent condition for being an $I$-Ulrich module \cite[Theorem 4.6]{Dao-Maitra-Sridhar}.
\begin{definition}
    Let $R$ be a one dimensional Cohen-Macaulay local ring and $I$ a regular ideal. Then a maximal Cohen-Macaulay $R$-module is called \textit{$I$-Ulrich} if there is an isomorphism $IM \cong M$. We denote by $\mathrm{Ul}_I(R)$ the category of $I$-Ulrich modules.
\end{definition}
\begin{remark}\label{i-ulrich-modules}
    Let $R$ be a one dimensional Cohen-Macaulay local ring and $I$ be a regular ideal.
    \begin{enumerate}
        \item A module $M$ is $I$-Ulrich if and only if $M \in \CM(B(I))$ \cite[Theorem 4.6]{Dao-Maitra-Sridhar}. This has the consequence that if $M$ is an $I$-Ulrich module, then one has an inclusion $\tr(M) \subseteq b(I)$ and the converse holds when $M$ is reflexive.
        \item If $M$ is an $I$-Ulrich module, then for any $X \in \CM(R)$, the module $\Hom_R(M,X)$ is also an $I$-Ulrich module \cite[Lemma 4.15]{Dao-Maitra-Sridhar}.
        \item If $R$ is analytically unramified, then the normalisation $\overline{R}$ is $I$-Ulrich.
    \end{enumerate}
\end{remark}
Recall that a Cohen-Macaulay local ring $R$ is generically Gorenstein if for every minimal prime $p$ of $R$ the local ring $R_p$ is Gorenstein. In this case, the ring $R$ admits a canonical module $\omega$. In fact, the generically Gorenstein condition is equivalent to saying that the canonical module $\omega$ is isomorphic to an ideal. When we take $\omega$ to be an ideal we will refer to it as a canonical ideal.
\begin{remark}\label{omega-Ulrich-remarks}
    Let $R$ be a one dimensional generically Gorenstein local ring with canonical ideal $\omega$.
    \begin{enumerate}
        \item There is an equality $\reflexive(R) = \Omega \CM(R)$. That is, an $R$-module $M$ is reflexive if and only if it is the syzygy of a maximal Cohen-Macaulay module.
        \item An $R$-module $M$ is $\omega$-Ulrich if and only if there is an isomorphism $M^* \cong DM$ \cite[Corollary 4.27]{Dao-Maitra-Sridhar}.
        \item For $n \gg 0$, the ideal $\omega^n$ is $\omega$-Ulrich.
    \end{enumerate}
\end{remark}
The following theorem follows from \cite[Theorem 5.5]{Dao-Dey-Dutta}. We note that by the \textit{additive closure} $\add_R(X)$ of an $R$-module $X$, we mean the smallest subcategory that contains $X$ and is closed under direct summands, isomorphisms and finite direct sums.
\begin{theorem}\label{main-theorem}
     Let $R$ be a one dimensional Cohen-Macaulay local ring which is analytically unramified. Let $\omega$ be a canonical ideal. Then, the following are equivalent.
    \begin{enumerate}
        \item The blowup algebra $B(\omega)$ is the normalisation $\overline{R}$.
        \item There is an equality $b(\omega) = \co(R)$.
        \item There exists an $n$ such that $\omega^n \cong \co(R)$.
    \end{enumerate}
    These equivalent conditions imply the following.
    \begin{enumerate}
        \item[(4)] There is an equality $\mathrm{Ul}_{\omega}(R) = \add_R(\overline{R})$.
    \end{enumerate}
    When $R$ is a Henselian domain, all four conditions are equivalent.
\end{theorem}
We end this section with our definition and an example.
\begin{definition}
    Let $R$ be a one dimensional Cohen-Macaulay local ring with canonical ideal $\omega$. Assume that $R$ is analytically unramified and is not regular. We say that $R$ has \textit{minimal canonical conductor} if the equivalent conditions of Theorem \ref{main-theorem} hold for $R$.
\end{definition}
\begin{example}\label{one-step-normal-example}
    Let $n \geq 3$ be a positive integer and $R = k[[t^n, t^{n+1}, \ldots, t^{2n-1}]]$. Then, by direct computation we can see that $\co(R)=\m$ and therefore one has equalities
    \begin{align*}
        \co(R) = b(\omega) = \tr(\omega) = \m.
    \end{align*}
    Therefore, $R$ has minimal canonical conductor. Rings with maximal conductor have been considered by several authors in the last decade. For instance, they appear in Faber's work \cite{Faber} as \textit{one-step normal} rings. In the work of Dao-Dey-Dutta \cite{Dao-Dey-Dutta}, such rings are exactly one dimensional \textit{Ulrich split} rings. In the language of Herzog-Kumashiro-Stamate \cite{Herzog-Kumashiro-Stamate}, they are \textit{far-flung Gorenstein and nearly Gorenstein}. Recently, such rings also appeared in the work of Geroldinger-Yan-Zhong as \textit{conductor domains} \cite{Geroldinger-Yan-Zhong}.
\end{example}
In Example \ref{one-step-normal-example}, we have seen that the canonical conductor $b(\omega)$ equals the canonical trace ideal $\tr(\omega)$. As we always have the inclusions $\co(R) \subseteq b(\omega) \subseteq \tr(\omega)$ for a one dimensional analytically unramified local ring, the condition $\tr(\omega) = \co(R)$ implies the condition $b(\omega) = \co(R)$. Hence, far-flung Gorenstein rings have minimal canonical conductor. However, not all rings with minimal canonical conductor are far-flung Gorenstein. Yet, two classes may sometimes agree. We start with a definition due to Kumashiro \cite{Kumashiro}.
\begin{definition}
    Let $R$ be a generically Gorenstein Cohen-Macaulay local ring. The infimum of all nonnegative integers $n$ such that there exists a canonical ideal $\omega$ and an almost reduction $(a)$ of $\omega$ with $\omega^{n+1} = a\omega^n$ is called the \textit{canonical reduction number} of $R$ and is denoted by $\mathrm{can.red}(R)$.
\end{definition}
As mentioned above, a one dimensional analytically unramified Cohen-Macaulay local ring is reduced and therefore it is generically Gorenstein.
\begin{proposition}
    Let $R$ be a one dimensional Cohen-Macaulay local ring which is analytically unramified. The following are equivalent.
    \begin{enumerate}
        \item $R$ has minimal canonical conductor and canonical reduction number 2.
        \item $R$ is far-flung Gorenstein.
    \end{enumerate}
\end{proposition}
\begin{proof}
    The inclusion $b(\omega) \subseteq \tr(\omega)$ is an equality if and only if there exists an isomorphism $\tr(\omega) \cong \omega^*$ by \cite[Proposition 4.7]{Dao-Lindo}. And such an isomorphism holds if and only if $R$ has canonical reduction number at most 2 by \cite[Theorem 2.13]{Kumashiro}. This shows that if $R$ has minimal canonical conductor and canonical reduction number 2, then it is far-flung Gorenstein. On the other hand, as mentioned above all far-flung Gorenstein rings have minimal canonical conductor. Moreover, by \cite[Theorem 2.5]{Herzog-Kumashiro-Stamate}, we see that $R$ has canonical reduction number 2 when it is far-flung Gorenstein. This finishes the proof.
\end{proof}
We finish this section with an observation on the endomorphism ring $\m \colon \m$. This ring is often useful when one studies reflexive modules over $R$ and is far-flung Gorenstein whenever $R$ is. The latter assertion was proved by Herzog-Kumashiro-Stamate in \cite[Theorem 3.3]{Herzog-Kumashiro-Stamate} and the proof of the proposition below works similarly by making use of Claim 1 in the proof of \textit{op. cit}.
\begin{proposition}
    Let $(R, \m)$ be a one dimensional analytically unramified local Henselian domain. If $R$ has minimal canonical conductor, then the ring $B = \m \colon \m$ also has minimal canonical conductor.
\end{proposition}

\section{Reflexive and Gorenstein birational extensions}
Our definition of a ring with minimal canonical conductor via the chain of ideals
\begin{align*}
\co(R) \subseteq b(\omega) \subseteq \tr(\omega) \subseteq \m
\end{align*}
suggests that such a ring should be considered as far away as possible from being almost Gorenstein. However, this claim uses a specific characterisation of one dimensional almost Gorenstein rings, namely \cite[Proposition 7.1]{Dao-Maitra-Sridhar}. Another characterisation of a one dimensional almost Gorenstein local ring $(R,\m)$ is via the endomorphism ring $E=\m \colon \m \cong \End_R(\m)$. Precisely, in \cite[Theorem 5.1]{Goto-Matsuoka-Phuong}, Goto-Matsuoka-Phuong proves that the following are equivalent for a one dimensional Cohen-Macaulay local ring $(R,\m)$:
\begin{enumerate}
    \item $E$ is a Gorenstein ring.
    \item $R$ is an almost Gorenstein ring with minimal multiplicity.
\end{enumerate}
In this section, we will explore this point of view. We start with some important remarks.
\begin{remark}\label{remark-for-Gorenstein-module-finite-birational-reflexive-extension}
    Let $(R,\m)$ be a one dimensional Cohen-Macaulay local ring which is generically Gorenstein and put $E = \End_R(\m)$ as above. Then, both $\m$ and $E$ are reflexive $R$-modules, see for instance \cite[Lemma 2.5]{Kobayashi}. Assume that $R$ is not regular. Then, by \cite[Lemma 2.2]{Kobayashi-self-dual}, for any module-finite reflexive birational extension $S$ of $R$, one has a chain of inclusions
    \begin{align*}
    R \subsetneq E \subseteq S.
    \end{align*}
    That is, $E$ is the minimal module-finite reflexive birational extension of $R$, in a sense, other than $R$ itself. Now, for a one dimensional Cohen-Macaulay local ring $R$, consider the following assertions.
    \begin{enumerate}
        \item Any module-finite reflexive birational extension of $R$ is Gorenstein.
        \item $R$ is an almost Gorenstein ring with minimal multiplicity.
    \end{enumerate}
\end{remark}
Then (1) implies (2) by the above discussion.
\begin{remark}
    Let $(R,\m)$ be a one dimensional Cohen-Macaulay local ring with infinite residue field and assume that it is not regular. In this case, $\m$ is not a principal ideal and therefore $\m$ is a regular proper trace ideal, see for instance \cite[Proposition 2.8]{Herzog-Rahimbeigi}. By \cite[Proposition 2.5]{Ooishi}, we have that $R$ has minimal multiplicity if and only if there exists an isomorphism $\m \cong \m^{-1} \cong \m^*$. And by \cite[Proposition 3.3]{Dao-Lindo}, the latter isomorphism is equivalent to saying that $\m$ is a stable trace ideal. Finally, let $\omega$ be the canonical module of $R$ and assume that $R$ is not Gorenstein so that $\omega$ is not free. If $R$ is almost Gorenstein, then by \cite[Theorem 7.1]{Dao-Maitra-Sridhar}, we have $b(\omega) = \m$. Now, consider the following two statements for a one dimensional Cohen-Macaulay local ring $R$ with canonical module $\omega$:
    \begin{enumerate}
        \item Any module-finite reflexive birational extension of $R$ is Gorenstein.
        \item $R$ is an almost Gorenstein ring and $b(\omega)$ is stable.
    \end{enumerate}
    Then, combining the above discussion with Remark \ref{remark-for-Gorenstein-module-finite-birational-reflexive-extension}, we see that (1) implies (2).
\end{remark}
Let $R$ be a one dimensional analytically unramified local ring. Then, the normalisation $\overline{R}$ of $R$ is a reflexive birational extension which is Gorenstein. We can now prove that this is the only such birational extension when $R$ has minimal canonical conductor.
\begin{theorem}\label{new-main-theorem}
     Let $R$ be a one dimensional Cohen-Macaulay local ring which is analytically unramified. Assume $R$ is not regular. Consider the two conditions.
    \begin{enumerate}
        \item $R$ has minimal canonical conductor.
        \item If $R \subseteq S \subseteq \overline{R}$ is a Gorenstein birational extension that is reflexive as an $R$-module, then $S = \overline{R}$.
    \end{enumerate}
    Then, (1) implies (2). If $b(\omega)$ is stable, then (2) implies (1).
\end{theorem}
\begin{proof}
        Assume that $R$ has minimal canonical conductor. Let $S$ be a Gorenstein birational extension of $R$ which is reflexive as an $R$-module. Then, by \cite[Theorem 7.10]{Dao-Maitra-Sridhar}, we know that $\co(R,S)$ is $\omega$-Ulrich. Therefore, $\co(R,S)$ belongs to $\add_R(\co(R))$ and is reflexive. Hence $\co(R,S) = \tr(\co(R,S))= \co(R)$ which shows that $S = \overline{R}$ since $S$ and $\overline{R}$ are both birational extensions of $R$ with $S \subseteq \overline{R}$.

        On the other hand, assume (2) and that $b(\omega)$ is stable. We know that $B(\omega)$ is an $\omega$-Ulrich module. So, it is a reflexive birational extension of $R$. We know that $\tr(B(\omega))=\co(R, B(\omega)) = b(\omega)$ is $\omega$-Ulrich, as well. Moreover, by our assumption, it is also stable. Thus, by \cite[Theorem 7.10]{Dao-Maitra-Sridhar}, we see that $B(\omega)$ is Gorenstein. Hence, by our assumption, we get that $B(\omega) = \overline{R}$ and therefore, we are done by Theorem \ref{main-theorem}.
\end{proof}
We finish this section with an application to Arf local rings. The following definition and remarks on Arf local rings can be found in Lipman's original work on Arf rings \cite{Lipman}.
\begin{definition}
    Let $R$ be a one dimensional Noetherian semi-local Cohen-Macaulay ring.
    \begin{enumerate}
        \item We define a chain of blowup algebras inductively by letting $R_0 = R$ and $J_0 = J(R)$ (where $J(R)$ denotes the Jacobson radical of $R$) and setting $R_{n+1} := \End_{R_n}(J_n)$ where $J_n = J(R_n)$. If $\m$ is a maximal ideal of some $R_n$, then the local ring $(R_n)_\m$ is said to be \textit{infinitely near} to $R$.
        \item  We say that $R$ is an \textit{Arf ring} if every integrally closed ideal containing a nonzerodivisor has a principal reduction and whenever $x \in R$ is nonzerodivisor and $y,z \in R$ arbitrary elements with $y/x$ and $z/x$ in $\overline{R}$, we must have $yz/x \in R$.
    \end{enumerate}
\end{definition}
\begin{remark}
    The following are equivalent:
\begin{enumerate}
    \item $R$ is an Arf ring,
    \item Every integrally closed ideal containing a nonzerodivisor is stable,
    \item If $B$ is any local ring infinitely near to $R$, then $B$ has minimal multiplicity.
\end{enumerate}
It is clear from this that if $R = R_0$ is Arf, then $R_n$ is Arf for any $n \geq 0$. Moreover, if $R$ is a local ring with minimal multiplicity and $R_1$ is Arf, then $R$ is also Arf.
\end{remark}
\begin{remark}
    By \cite[Section 7]{Dao-Lindo} or \cite[Theorem B]{Dao}, any regular trace ideal over an Arf ring (under mild conditions that we satisfy) is  stable. In particular, in the setting of Arf rings, we know that $b(\omega)$ is stable. 
\end{remark}
Let $R$ be a complete local domain of dimension one and assume that it is analytically unramified. Then, we have a chain
\begin{align*}
R = R_0 \subseteq R_1 \subseteq R_2 \subseteq \cdots \subseteq R_{n-1} \subseteq R_n = \overline{R}
\end{align*}
consisting of complete local domains which are infinitely near to $R$. By \cite[Theorem A]{Dao}, the $R_i$'s give a complete list of indecomposable reflexive $R$-modules up to isomorphism when $R$ is Arf . Now, assume that $R$ is Arf. If $R \subseteq S \subseteq \overline{R}$ is a birational extension that is reflexive as an $R$-module, then $S = R_i$ for some $i = 0, \ldots, n$. Hence, any birational extension $S \subseteq \overline{R}$ which is reflexive as an $R$-module is Arf. On the other hand, by \cite[Corollary 4.3]{Barucci-Froberg-value-semigroups}, $S$ is Gorenstein and Arf if and only if $S$ has multiplicity at most 2. This gives us the following characterisation by applying Theorem \ref{new-main-theorem}.
\begin{theorem}
    Let $R$ be a complete local Arf domain and assume that it is analytically unramified. The following are equivalent.
    \begin{enumerate}
        \item $R$ has minimal canonical conductor.
        \item For any birational extension $R\subseteq S \subseteq \overline{R}$ that is reflexive as an $R$-module, we have either $S = \overline{R}$ or $e(S) \geq 3$.
        \item There does not exist an $A_n$-singularity that is infinitely near to $R$.
    \end{enumerate}
\end{theorem}

\section{Numerical semigroups with minimal canonical conductor}
In this section, we classify numerical semigroup rings with minimal canonical conductor. By a numerical semigroup $H$, we mean a subsemigroup of $\mathbb{N}$ with $\mathbb{N}\setminus H$ finite. This means that there exists a largest number $f$ which does not belong to $H$. This $f$ is called the \textit{Frobenius} number of $H$. It has the property that $f + h \in H$ for any $h \in H$ with $h \neq 0$. This property can also hold for other numbers which do not belong to $H$. We put
\begin{align*}
\mathrm{PF}(H):= \{x \in \mathbb{Z} \setminus H \colon x + h \in H \text{ for all } h \in H \text{ with } h \neq 0 \}
\end{align*}
and call an element of this set a \textit{pseudo-Frobenius} number. The \textit{Frobenius number} $f$ is the largest pseudo-Frobenius number and the \textit{conductor} $c$ is defined as $f+1$ so that for any natural number $h$ with $h \geq c$ we have $h \in H$.

Given a field $k$, we denote by $k[[H]]$ the subring of the polynomial ring $k[[t]]$ generated by $t^h$ such that $h \in H$. There exists a nice interplay between $H$ and $k[[H]]$. For instance, every numerical semigroup has a minimal set of generators. The number of minimal generators of $H$ coincides with the embedding dimension of $k[[H]]$. The smallest element in $H$ coincides with the multiplicity of $k[[H]]$. Hence, $k[[H]]$ is of minimal multiplicity if and only if $H$ is minimally generated by elements of the form $n< a_2< \cdots < a_n$.

We can also see the type of $R = k[[H]]$ from the semigroup data. It is given by the size of $\mathrm{PF}(H)$. More precisely, if $f$ is the Frobenius number of $H$, then the $R$-module
\begin{align*}
\omega = \sum_{x \in \mathrm{PF}(H)} Rt^{f-x}
\end{align*}
is a canonical module which is fractional ideal sitting between $R$ and $k[[t]] = \overline{R}$ \cite[Korollar 5]{Jager}.

\begin{theorem}\label{pseudo-frobenius-theorem}
    Let $k$ be a field and $H$ be a numerical semigroup with conductor $c \geq 3$. Then, the following are equivalent.
    \begin{enumerate}
        \item The numerical semigroup ring $k[[H]]$ has minimal canonical conductor.
        \item $c-2$ is not in $H$. 
    \end{enumerate}
\end{theorem}
\begin{proof}
    Denote by $f$ the Frobenius number of $H$ so that $f-1 = c-2$. Note that if $c-2$ is a pseudo-Frobenius number, then by definition $c-2 \notin H$. Now, assume that $c-2 \notin H$. Since $c \geq 3$, we have that $1 \notin H$. Therefore, for any $h \in H$, we have $c-2 + h \geq c $ and so we get $c-2+h \in H$. This means $c-2 \in PF(H)$. So, we proved that $c-2 \notin H$ if and only if $c-2 = f-1$ is a pseudo-Frobenius number.

    If $f-1$ is a pseudo-Frobenius number for $H$, then there exists a canonical fractional ideal $\omega = R + Rt + J$ where $J = R^{t_1} + \ldots + R^{t_s}$ for some $t_1, \ldots, t_s$. Note that $\omega^n \subseteq k[[t]] = \overline{R}$ for any positive integer $n$. On the other hand, we have
    \begin{align*}
    (R+R^t)^n = R + Rt + Rt^2 + \cdots + Rt^n \subseteq \omega^n.
    \end{align*}
    For large enough $n$ (we can choose $n \leq f$), we have $(R+Rt)^n = k[[t]] = \overline{R}$. Hence, we obtain an isomorphism $\omega^n \cong \overline{R}$.

    If $f-1$ is not a pseudo-Frobenius number, we the canonical fractional ideal is of the form $\omega = R + R{t^{m_1}} + \cdots + Rt^{m_s}$ for some $m_1, \ldots, m_s \geq 2$. Therefore, we have $t \notin \omega^n$ for any $n \geq 1$. This finishes the proof.
\end{proof}
\begin{corollary}\label{semigroup-corollary}
    Let $H$ be a numerical semigroup minimally generated by $n< a_2< \cdots < a_n$. Then, the ring $k[[H]]$ has minimal canonical conductor if and only if $a_{n} = a_{n-1} +1$.
\end{corollary}
\begin{proof}
    Since the minimal number of generators equals the smallest generator, $k[[H]]$ is of minimal multiplicity. Therefore, pseudo-Frobenius numbers of $H$ are given by $a_2 - n, a_3-n, \ldots, a_n - n$. Hence, the Frobenius number is $a_n-n$ and $a_n-n-1$ is a Frobenius number if and only if $a_{n-1} = a_n - 1$.
\end{proof}
\begin{example}
    Let $N$ be the semigroup generated by $n, n+1, n^2-n-1$ with $n \geq 3$ and $R = k[[N]]$ be the local semigroup ring generated by $N$. This ring was studied in \cite[Example 3.13]{Kumashiro}. It has a canonical fractional ideal $\omega = R + Rt$, it is nearly Gorenstein for all $n \geq 3$ and it is almost Gorenstein if and only if $n = 3$. We see from Theorem \ref{pseudo-frobenius-theorem} that
    \begin{align*}
        \co(R) = b(\omega) \subseteq \tr(\omega) = \m.
    \end{align*}
    Therefore, $R$ has minimal canonical conductor. The inclusion $b(\omega) \subseteq \tr(\omega)$ is strict if and only if $n > 3$. Hence this ring, while being nearly Gorenstein, is far from being almost Gorenstein for $n > 3$.
\end{example}

\section{Cohomology annihilators}

The cohomology annihilator ideal of a Noether algebra was introduced by Iyengar and Takahashi \cite{iyengar-takahashi}. It has strong connections with the singularities of a commutative ring, generation of module categories and dimension of singularity categories. We begin this section with its definition.

\begin{definition}
    Let $R$ be a commutative Noetherian ring. The \textit{cohomology annihilator ideal}, denoted by $\ca(R)$, of $R$ is defined as the set of all elements $r \in R$ such that \begin{align*}
        r \cdot \Ext_R^i(M,N) = 0
    \end{align*}
    for all $i \gg 0$ and every $M, N \in \module R$.
\end{definition}
To study the cohomology annihilator ideal, we look at the \textit{stable annihilator ideal} for each individual module. We will define it here and give some basic properties that we need.
\begin{definition}
    Let $R$ be a commutative Noetherian local ring and $M,N$ be finitely generated $R$-modules.
    \begin{enumerate}
        \item Denote by $P(M,N)$ the set of all morphisms from $M$ to $N$ that factors through a projective $R$-module. This is a submodule of $\Hom_R(M,N)$. We denote the corresponding quotient module by $\sHom_R(M,N)$. When $M = N$, we use the notation $\sEnd_R(M)$.
        \item We denote by $\sann_R(M)$ the annihilator of the $R$-module $\sEnd_R(M)$ and call it the \textit{stable annihilator} of $M$.
        \item Given a projective $R$-module $P$ and a surjective morphism $P \to M \to 0$, we call the kernel a \textit{syzygy} of $M$ and denote it by $\Omega_R M$. This is uniquely determined up to projective summands by Schanuel's Lemma. When it is clear from the context, we drop the subscript and use $\Omega M$.
    \end{enumerate}
\end{definition}
\begin{remark}
    There exists equalities 
    \begin{align*}
        \sann_R(M) = \ann_R \Ext_R^1(M, \Omega M) = \bigcap_{N \in \module R} \bigcap_{i > 0} \ann_R \Ext_R^i(M,N)
    \end{align*}
    by the discussion at the end of \cite[Section 2]{iyengar-takahashi}.
\end{remark}
\begin{remark}\label{stable-annihilator-of-conductor}
    Let $R$ be one dimensional Cohen-Macaulay local ring. Then, by \cite[Corollary 4.2]{Esentepe}, we have an equality 
    \begin{align*}
    \co(R) = \sann_R(\co(R)).
    \end{align*}
    
\end{remark}

\begin{remark}\label{cohomology-annihilator-lemma}
    Let $R$ be a commutative Noetherian ring of finite Krull dimension $d$. Then, there exists an $s \gg 0$ such that 
    \begin{align*}
    \ca(R) &= \bigcap_{M,N \in \module R} \ann_R \Ext_R^{s+1}(M, N) \\ &= \bigcap_{M,N \in \module R} \ann_R \Ext_R^1(\Omega^s M, N) \\  &= \bigcap_{M \in \module R} \sann_R (\Omega^s M).
    \end{align*}
    In particular, for any finitely generated $R$-module $M$, we have an inclusion $\ca(R) \subseteq \sann_R(\Omega^s M)$ for $s \gg 0$. 
\end{remark}
\begin{remark}\label{Wangs-theorem}
    Let $R$ be a one dimensional reduced complete Noetherian local ring. Wang proves in \cite{Wang} that there exists an inclusion $\co(R) \subseteq \ca(R)$. Combining this with the previous remark, we obtain the inclusions 
    \begin{align*}
        \co(R) \subseteq \ca(R) \subseteq \sann_R (\Omega^s \co(R)).
    \end{align*}
\end{remark}
We are now ready to move on with the Lindo-Lyle-Maitra family of numerical semigroup rings.
\begin{definition}
    Let $k$ be a field and $n,s$ be integers with $n \geq 2, s \geq 1$. Define $H_{s,n}$ to be the numerical semigroup generated by \begin{align*}
        n, sn+1, sn+2, \ldots, sn+n-1.
    \end{align*}
    We define $R_{s,n}$ to be the corresponding numerical semigroup ring.
\end{definition}
Note that if $n = 2$ or $s = 1$, then the ring is almost Gorenstein.
\begin{proposition}
 For each $n \geq 3$ and $s \geq 2$ , the ring $R_{s,n}$ has minimal canonical conductor.
\end{proposition}
\begin{proof}
    $R_{s,n}$ has multiplicity $n$ and is generated by $n$ elements: it has minimal multiplicity. By Corollary \ref{semigroup-corollary}, it has minimal canonical conductor as the two largest generators of its semigroup are $sn+n-2$ and $sn+n-1$.
\end{proof}
\begin{proposition}\label{proposition-direct-sum-family-syzygy}
    Let $R= R_{s,n}$ for some $n,s$. Then, there exists an isomorphism
    \begin{align*}
        \Omega_R (\co(R)) \cong \co(R)^{\oplus m}
    \end{align*}
    for some positive integer $m$.
\end{proposition}
\begin{proof}
    By \cite[Theorem 1.3]{Lindo-Lyle-Maitra}, we have
    \begin{align*}
    \tr(\Omega_R (\co(R))) = \co(R) = b(\omega)
    \end{align*}
    where the second equality follows from having minimal canonical conductor. As $\Omega_R (\co(R))$ is a reflexive module, this equality gives us that it is $\omega$-Ulrich. Since $R$ has minimal canonical conductor, this implies $\Omega_R (\co(R)) \in \add_R(\overline{R})$. Since $\overline{R}$ is indecomposable as an $R$-module and since $\overline{R} \cong \co(R)$, the result follows.
\end{proof}
\begin{theorem}
    For every  $n \geq 3$ and $s \geq 2$, there exists an equality
    \begin{align*}
        \ca(R_{s,n}) = \co(R_{s,n}).
    \end{align*}
\end{theorem}
\begin{proof}
Put $R= R_{s,n}$. The inclusion $\ca(R) \supseteq \co(R)$ is by Wang's theorem in Remark \ref{Wangs-theorem}. For the other inclusion, observe that we have
\begin{align*}
    \Omega_R^q(\co(R)) \cong \co(R)^{\oplus t}
\end{align*}
for integers $q, t \gg 0$ by the previous proposition. Therefore, we have
\begin{align*}
    \co(R) = \sann_R(\co(R)) = \sann_R(\Omega_R^q(\co(R)))
\end{align*}
where the first equality is from Remark \ref{stable-annihilator-of-conductor}. This gives us, by Remark \ref{cohomology-annihilator-lemma}, that $\ca(R) \subseteq \co(R)$ which finishes the proof.
\end{proof}
\begin{remark}
    Let $\m_{s,n}$ denote the maximal ideal of $R_{s,n}$. Then, direct computations show that we have
    \begin{align*}
        \End_{R_{s,n}}(\m_{s,n}) \cong R_{s-1,n}.
    \end{align*}
    This gives us a tower of endomorphism rings
    \begin{align*}
        R_{s,n} \subseteq R_{s-1,n} \subseteq \cdots \subseteq R_{2,n} \subseteq R_{1,n} \subseteq \overline{R}.
    \end{align*}
    We know that $R_{1,n}$ is Gorenstein and hence $R_{2,n}$ is almost Gorenstein. But for $s \geq 3$, our ring $R_{s,n}$ is not almost Gorenstein. So, we have a genuinely new family of rings whose cohomology annihilator ideal coincides with its conductor ideal. Moreover, these rings are all far from being almost Gorenstein.
\end{remark}
We finish the paper with another application related to annihilators of Ext-modules: The Auslander-Reiten Conjecture for commutative Noetherian rings states the following \cite{Auslander-Reiten,Auslander-Ding-Solberg}: Given a finitely generated module $M$ over a commutative Noetherian ring $R$. If there exists a vanishing $\Ext_R^i(M, M \oplus R) = 0$ for all $i \geq 1$, then $M$ is projective. But similar to our philosophy in the introduction, in \cite{Esentepe-Auslander-Reiten}, we considered the failure of this vanishing via annihilator ideals. In particular, we have the following definition of an \textit{Auslander-Reiten annihilator}:
\begin{align*}
\arann_R(M) := \bigcap_{i \geq 1} \ann_R \Ext_R^i(M, M \oplus R).
\end{align*}
The Conjecture states that whenever $\arann_R(M) = R$ we must have $\sann_R(M) = R$ which is equivalent to saying that $M$ is free. The main question in \cite{Esentepe-Auslander-Reiten} is the study of the equality $\sann_R(M) = \arann_R(M)$ in general.
\begin{theorem}
    Let $R = R_{s,n}$ so that $\overline{R} = k[[t]]$. Then, for any $i \geq 1$, we have 
    \begin{align*}
    \co(R) = \ann_R \Ext_R^i(\overline{R}, \overline{R}).
    \end{align*}
    In particular, we have 
    \begin{align*}
    \sann_R (\overline{R}) = \arann_R(\overline{R}).
    \end{align*}    
\end{theorem}
\begin{proof}
    There exists a well-known isomorphism $\co(R) \cong \overline{R}$. Therefore, for every $i \geq 1 $ there exists a $j > 0$ and an isomorphism 
    \begin{align*}
    \Omega^{i-1} \overline{R} \cong \overline{R}^{\oplus j}
    \end{align*}
    by Proposition \ref{proposition-direct-sum-family-syzygy}. By additive properties of Ext and by properties of annihilator ideals, we get 
    \begin{align*}
    \ann_R \Ext_R^i(\overline{R}, \overline{R}) &= \ann_R \Ext_R^1 (\Omega^{i-1}\overline{R}, \overline{R}) \\ 
    &= \ann_R \Ext_R^1 (\overline{R}^{\oplus j}, \overline{R}) \\
    &= \ann_R \Ext_R^1 (\overline{R}, \overline{R}^{\oplus m}) \\
    &= \ann_R \Ext_R^1 (\overline{R}, \Omega \overline{R}) \\
    &= \sann_R(\overline{R}).
    \end{align*}

\end{proof}

\bibliographystyle{alpha}
\bibliography{mcc.bib}
\end{document}